\documentclass[fleqn,final]{zzz}

\usepackage{mathrsfs}
\usepackage[pdftex]{hyperref}

\hypersetup{
	colorlinks = true,
	urlcolor = blue,
	linkcolor = red,
	citecolor = green
}

\makeatletter
\def\eqnarray{\stepcounter{equation}\let\@currentlabel=\theequation
	\global\@eqnswtrue
	\tabskip\@centering\let\\=\@eqncr
	$$\halign to \displaywidth\bgroup\hfil\global\@eqcnt\z@
	$\displaystyle\tabskip\z@{##}$&\global\@eqcnt\@ne
	\hfil$\displaystyle{{}##{}}$\hfil
	&\global\@eqcnt\tw@ $\displaystyle{##}$\hfil
	\tabskip\@centering&\llap{##}\tabskip\z@\cr}

\def\endeqnarray{\@@eqncr\egroup
	\global\advance\c@equation\m@ne$$\global\@ignoretrue}

\def\@yeqncr{\@ifnextchar [{\@xeqncr}{\@xeqncr[5pt]}}
\makeatother

\begin{document}

\renewcommand{\PaperNumber}{***}

\FirstPageHeading

\ShortArticleName{
Analytic properties of some basic hypergeometric-Sobolev-type orthogonal polynomials}

\ArticleName{Analytic properties of some basic 
	hypergeometric-Sobolev-type orthogonal polynomials}

\Author{Roberto S.~Costas-Santos\,$^\dag\!\!\ $,
	Anier Soria-Lorente\,$^\ddag{}$}

\AuthorNameForHeading{R. S.~Costas-Santos, Anier Soria-Lorente}
\Address{$^\dag$~Departamento de F\'isica y Matem\'{a}ticas,
	Universidad de Alcal\'{a}
, Alcal\'{a} de Henares, Spain
} 
\URLaddressD{
	\href{http://www.rscosan.com}
	{http://www.rscosan.com}
}
\EmailD{rscosa@gmail.com} 
\Address{$^\ddag$ Department of Basic Sciences, Granma University,  
Bayamo, Cuba}
\EmailD{asorial@udg.co.cu}


\ArticleDates{Received 2 June 2018 in final form ????; Published online 31 August 2018}

\Abstract{In this contribution we consider sequences of 
	monic polynomials orthogonal with respect to a 
	Sobolev-type inner product
	\[
	\langle f,g \rangle _{S}:=
	\langle {\bf u}, f g\rangle +N (\mathscr D_q f)(\alpha)
	(\mathscr D _{q}g)(\alpha),\qquad \alpha\in \mathbb R,
	\quad N\ge 0,
	\]
	where $\bf u$ is a $q$-classical linear functional and 
	$\mathscr D _{q}$ is the $q$-derivative operator. 
	
	We obtain some algebraic properties of these polynomials 
	such as an explicit representation, a five-term recurrence 
	relation as well as a second order linear $q$-difference 
	holonomic equation fulfilled by such polynomials.
	
	We present an analysis of the behaviour of  its zeros 
	as a function of the mass $N$. In particular, we 
	obtain the exact values of $N$ such that the 
	smallest (respectively, the greatest) zero of 
	the studied polynomials is located outside of the 
	support of the measure.
	
	We conclude this work considering two examples.
}

\Classification{33D45, 05A30, 39A13}

\section{Introduction}
The study of polynomial sequences orthogonal with 
respect to an inner product  involving differences 
was started in two papers 
\cite{bav1,bav2}. H. Bavinck considered the inner product
\begin{equation} 
\langle p, q\rangle=\int_{\mathbb R} p(t) q(t) d\mu(t)
+\lambda (\Delta p)(c)(\Delta q)(c), \label{1:1}
\end{equation}
where $p$, $q$ are polynomials with real coefficients, 
$c \in \mathbb R$, $\mu$ is a distribution function with 
infinite support such that $\mu$ has no points of 
increase in the interval $(c, c+1)$, $\lambda\in \mathbb R_+$, 
and where $(\Delta p)(c)=p(c+1)-p(c)$ denotes the 
forward difference operator.

Later on, in \cite{bav3} the authors obtained a 
difference operator of  infinite order for which 
these orthogonal polynomials (called 
Sobolev-type Meixner polynomials) are eigenfunctions. 
The name Sobolev-type is justified from the analogy 
with the case
\[
\langle p, q\rangle=\int_{\mathbb R} p(t) q(t) d\mu(t)
+M  p'(c)q'(c),
\]
which has been widely considered in the literature 
(e.g. see the survey in Sobolev polynomials 
\cite{Fin1}). 

Taking all this into account, the main idea of this 
paper is to obtain similar results for the polynomials 
sequence orthogonal with respect to a $q$-analogue of 
\eqref{1:1}, i.e. orthogonal with respect to the 
Sobolev-type inner product
\[
\langle f,g \rangle _{S}:=
\langle {\bf u}, f g\rangle +N (\mathscr D_q f)(\alpha)(\mathscr 
D _{q}g)(\alpha ),
\]
where  $\bf u$ is a $q$-classical linear functional and 
$N, \alpha \in \mathbb{R}$. 

The structure of the paper is the following: In Section 2, 
we introduce some notation and  results we need 
to prove some of the results throughout the paper. 
In Section 3, we define the discrete  Sobolev-type 
polynomials and present some algebraic, and 
analytical, results  about these polynomials.  
And in Section 4 we apply the obtained results 
to the Al-Salam-Carlitz I polynomials as well as  
to the Stieltjes--Wigert polynomials.
\section{Basic definitions, notations and results}
Let  $\mathbb P$ denote the vector space of univariate, 
complex-valued, polynomials, and let $\mathbb P'$ 
denote its algebraic dual space. 
We denote by $\langle {\bf u}, p\rangle$ the 
{\em duality bracket} for $\bf u\in \mathbb  P'$ 
and $p\in \mathbb P$, and $u_n=\langle {\bf u}, 
x^n\rangle$ with $n\ge 0$ are the canonical 
moments of {\bf u}.

\begin{definition} \cite{chi} A linear functional 
$\bf u$ is said to be quasi-definite  if the principal 
submatrices of the infinite Hankel matrix associated 
with the sequence of the moments $(u_n)_n$ 
of the linear functional $\bf u$, i.e.  $H_k=\left(u_{i+k}\right)_{i,j=0}^k$,
are non-singular. 
\end{definition}

We are going to consider the quasi-definite linear functional $\bf u$. 
Therefore, there exists a sequence of monic polynomials 
$(p_n)_n$ with $\deg p_n=n$, orthogonal with respect to $\bf u$, i.e. 
\[
\langle {\bf u}, p_n p_{m}\rangle = k_n \delta_{n,m}, 
\qquad k_n=\langle {\bf u}, p^2_n\rangle\ne 0.
\]
Such a sequence is said to be a monic orthogonal polynomial 
sequence (MOPS) associated with the linear functional $\bf u$.
\begin{definition}

Let $\bf u$ be a linear functional and let $p$ be a fixed polynomial. 
We define the linear functional $p\,{\bf u}$ as follows:
\[
\langle p\, {\bf u}, r\rangle:=\langle {\bf u}, p\, r\rangle,
\quad r\in \mathbb P.
\]
\end{definition}
We also need to introduce the concept of quasi-orthogonality that 
is weaker than the concept of orthogonality.
\begin{definition} 
Given a linear functional $\bf u$. Let $p_n$ be a polynomial of 
degree $n\geq r$. If $p_n$ satisfies  the conditions
\[ 
\langle {\bf u}, x^{j} p_n\rangle=
\begin{cases}
\displaystyle0, & j=0,1,\ldots, n-r-1, \\ 
\displaystyle\neq 0, & j=n-r,
\end{cases}
\] 
then $p_n$ is said to be quasi-orthogonal of order $r$ with respect 
to the linear 
functional $\bf u$.

Moreover, if there exists an integral representation of $\bf u$ as
follows
\[
\langle {\bf u}, p\rangle=\int_I p(x)\, d\mu(x),\quad I\subseteq\mathbb R,
\]
then $p_n$ is said to be quasi-orthogonal of order $r$ with respect 
to the measure $d\mu(x)$.
\end{definition}

The next definitions are related with the $q$-polynomials located 
in the Hahn class. In fact, we assume along the paper $0<q<1$. 
\begin{definition}
The $q$-derivative or the Euler--Jackson $q$-difference operator 
 ${\mathscr D}_{q}$ is defined as follows:
\[
({\mathscr D}_{q}f)(x):= \left\{\begin{array}{cl}
\displaystyle \frac{f(qx)-f(x) }{(q-1)x} &
\text{if}\ x\ne 0 \wedge q\ne 1, \\[3mm]
f'(x) & \text{if}\ x=0 \vee q=1.\end{array}\right.
\]
\end{definition}
\begin{definition}
Given a linear functional $\bf u$. We say  that $\bf u$ is $q$-classical if it 
fulfils the Pearson-type distributional difference equation 
\[
{\mathscr D_q}[\phi(s){\bf u}]=\psi(s){\bf u}, \quad \phi, \psi\in \mathbb P,
\]
where $\deg \phi\le 2$, $\deg \psi=1$.
The corresponding MOPS associated with $\bf u$ is said to be a 
$q$-classical discrete MOPS, also called (monic) $q$-polynomials.
\end{definition}
\begin{remark}
Observe that these functionals $\bf u$ 
usually have the form 
\[
\langle {{\bf u}}, P\rangle = \left\{ \begin{array}{ll} 
\displaystyle \int_{s_0}^{s_1} P(x) \rho(x)\,  d_qx, & \ \mbox{\small
Al-Salam-Carlitz I, discrete $q$-Hermite I,}
\\[0.5cm]
\displaystyle \int_{a}^{b} P(x) \rho(x) dx, & \ \mbox{\small
Stieltjes--Wigert.}
\end{array} \right.
\]
etc., where $\rho$ is a weight function satisfying the following
difference equation of Pearson-type
\[
\Delta [\phi(s) \rho(s)] \!=\!\psi(s)\rho(s) \Delta x(s-1/2)  \iff
\frac{\rho(s+1)}{\rho(s)}\!=\! \frac{\phi(s)\!+\!\psi(s)\Delta x(s-1/2)}{\phi(s+1)}.
\]
\end{remark}
\begin{definition} 
The Jackson $q$-integrals (see \cite{Gasper, koekoeketal}) are defined 
by
\[
\int_0^a f(t)d_qt := a(1-q) \sum_{n=0}^\infty f(q^n a) q^n,
\]
\[
\int_a^0 f(t)d_qt := -a(1-q) \sum_{n=0}^\infty f(q^n a) q^n,
\]
if $a>0$ and $a<0$, respectively. 
So, we have
\[
\int_a^b f(t) d_q t = \int_0^b f(t)d_qt \,- \int_0^a f(t)d_qt,
\]
and 
\[
\int_a^b f(t) d_q t = \int_a^0 f(t)d_qt \,+ \int_0^b f(t)d_qt,
\]
when $0<a<b$ and $a<0<b$, respectively. 
Furthermore, we make use of the improper $q$-Jackson integral
\[
\int_0^\infty f(t)d_qt := (1-q) \sum_{n=-\infty}^\infty f(q^n ) q^n.
\]
\end{definition}
Examples of these polynomials are the big $q$-Jacobi, the big 
$q$-Legendre, the  big $q$-Laguerre, the Little $q$-Legendre, 
the Al-Salam-Carlitz I, and the discrete $q$-Hermite of type I 
polynomials \cite[pp. 438, 443, 478, 534, 547]{koekoeketal}, 
among others.  

The next definitions are related with the $q$-calculus framework: 
\begin{definition}
The $q$-number $[z]_{q}$, is 
defined  by
\[
[z] _{q}:= \frac{1-q^{z}}{1-q},\quad z\in \mathbb{C},
\]
\end{definition}
\begin{definition}
A $q$-analogue of the factorial of $n$ is defined by
\[ 
[0]_q!:=1, \quad [ n] _{q}!:= [n] _{q}[n-1] _{q}
\cdots [1]_{q}, \ n=1, 2, \dots
\] 
\end{definition}
\begin{definition}
A $q$-analogue of the Pochhammer symbol,  or shifted factorial, 
\cite{Gasper,koekoeketal} is defined by
\[ 
(a;q) _0:= 1, \quad (a;q)_n:= \prod_{j=0}^{n-1}(1-aq^{j}) , \ n=1,2,\ldots
\]
\[
(a;q)_\infty= \prod_{j=0}^\infty (1-aq^{j}), \quad |a|<1.
\]
Moreover, we will use the following notation 
\[ 
(a_{1},\ldots ,a_{r};q) _{k}:= \prod_{j=1}^r(a_{j};q) _{k}.
\] 
\end{definition}
\begin{definition}
Let $\left(a_{i}\right)_{i=1}^{r}$ and $\left(b_{j}\right)_{j=1}^{s}$
be complex numbers such that $b_{j}\neq q^{-n}$ with $
n\in \mathbb{N}_0$, for $j=1,2,\ldots, s$. 
The basic hypergeometric series, or $q$-hypergeometric, $_{r}\phi _{s}$ 
series with variable $z$ is defined by
\[
_{r}\phi _{s}(
a_{1}, \ldots ,a_{r}; b_{1}, \ldots ,b_{s} ; q,z ) 
:= \sum_{k=0}^\infty \frac{(a_{1},\ldots ,a_{r};q
) _{k}}{ (b_{1},\ldots ,b_{s};q) _{k}}((-1)^k q^{\binom{k}{2}}
)^{1+s-r}\frac{z^{k}}{(q;q) _{k}}.
\]
\end{definition}
To complete this section we present some useful results we need 
along the paper. 
\begin{proposition}  (Christoffel-Darboux formula). 
Let $\left( p_n\right)$ be a sequence of monic polynomials orthogonal 
with respect to the linear functional 
$\bf u$.  If we denote the $n$-th reproducing kernel by
\[
K_n(x,y):= \sum_{k=0}^{n-1}\frac{p_{k}(
x) p_{k}(y) }{\langle {\bf u}, p^2_k\rangle}.
\]
Then, for all $n\in \mathbb{N}$, 
\begin{equation}
K_n(x,y) =\frac{1}{\langle {\bf u}, p^2_{n-1}\rangle}\frac{
p_{n}(x) p_{n-1}(y) -p_{n}(y) p_{n-1}(x) }{x-y}.  \label{2:2}
\end{equation}
\end{proposition}

Taking into account the inner product we have considered, then it 
is natural to consider the partial $q$-derivatives of $K_n(x,y) $ we
will use the following notation:
\[ 
{\mathscr K}_{q,n}^{(i,j) }(x,y) :=\sum_{k=0}^{n-1}
\frac{({\mathscr D}_{q,x}^{i}\, p_{k})(x) %
({\mathscr D}_{q,y}^{j}\, p_{k})(y) }{\langle {\bf u}, p^2_k\rangle}.  
\] 
The last result we present in this section is a generalization of the 
reproducing  property of the kernel. 

Let $\pi(x)$ be a polynomial of degree $n-1$, it  can 
be written  in terms of elements of the polynomial sequence 
$\left(p_n\right)_n$, i.e., 
\[
\pi(x) =\sum_{k=0}^{n-1}\frac{\langle {\bf u}, \pi(x)\,p_{k}(x) \rangle }{
\langle {\bf u}, p^2_k\rangle}p_{k}(x).
\] 
Thus, we have
\[
({\mathscr D}_{q,y}^{j} \pi)(y) =\sum_{k=0}^{n-1} \frac{\langle {\bf u}, 
\pi(x) p_{k}(x) \rangle }{\langle {\bf u}, p^2_k\rangle}({\mathscr D}_q^{j}p_{k})(y).
\] 
Then, using the fact that
\begin{eqnarray*}
\langle {\bf u}, {\mathscr K}_{q,n}^{(0,j) }(x,y) \pi(x) \rangle &=& 
\sum_{k=0}^{n-1}\langle {\bf u}, \frac{ p_{k}(x)({\mathscr D}_q^{j}p_{k})(y))}
{\langle {\bf u}, p^2_k\rangle} \pi(x) \rangle \\
&=&\sum_{k=0}^{n-1}\frac{\langle {\bf u}, \pi(x)  p_{k}(x) \rangle}
{\langle {\bf u}, p^2_k\rangle}({\mathscr D}_q^{j} p_{k})(y).
\end{eqnarray*}
we obtain the identity
\[
\langle {\bf u}, {\mathscr K}_{q,n}^{(0,j)}(x,y) \pi(x) \rangle
=({\mathscr D}_{q,y}^{j} \pi)(y).
\]
Observe that for $j=0$ one has the reproducing property of the 
kernel, i.e.,
\[ 
\langle {\bf u},  K_n(x,y) \pi(x)\rangle =\pi(y).
\] 
We also need to introduce the following result 
(see \cite[Lemma 1]{BDR-jcam02} or  \cite[Lemma 3]{DMR-anm10})
about the  behaviour of the zeros of a polynomial $f(x)=h_n(x)+cg_n(x)$, 
that is a linear combination of two polynomials of the same degree.
\begin{lemma} \label{Lemm:2.1} 
Let $h_n(x)=H(x-x_{1})\cdots (x-x_n)$ and $
g_n(x)=G(x-y_{1})\cdots (x-y_n)$ be two polynomials with real and simple
zeros, with $H, G>0$.
If 
\[
y_{1}<x_{1}<\cdots <y_n<x_n,
\]
then, for any  $c>0$, the polynomial  $f(x)=h_n(x)+cg_n(x)$
has $n$ real zeros, namely $\eta _{1}\le\cdots \le \eta _n$, 
which interlace with the zeros of $h_n(x)$ and $g_n(x)$ as follows 
\[
y_{1}<\eta _{1}<x_{1}<\cdots <y_n<\eta _n<x_n.
\]  
Moreover, each $\eta _{k}=\eta _{k}(c)$ is a decreasing function of 
$c$ and, for each $k=1,\ldots, n$, we have
\begin{equation*}
\lim_{c\to \infty}\eta _{k}(c)=y_{k}\quad \text{and}\quad
\lim_{c\to \infty}c\big(\eta _{k}(c)-y_{k}\big)=\dfrac{-h_n(y_{k})}{
g_n^{\prime }(y_{k})}.
\end{equation*}
\end{lemma}
\section{The discrete Sobolev-type polynomials }
We start this section introducing the Sobolev-type inner product
\begin{equation}
\langle f,g \rangle _{S}:=\langle {\bf u}, f g\rangle 
+N (\mathscr D_q f)(\alpha)(\mathscr D _{q}g)(\alpha ),  
\quad N, \alpha\in \mathbb{R}. \label{3:4}
\end{equation}%

We denote by $\left(s^{(\alpha)}_n(N;x)\right)$ the sequence 
of monic polynomials, orthogonal with respect to the inner product 
(\ref{3:4}).  
These polynomials are said to be basic hypergeometric-Sobolev-type 
orthogonal polynomials.

\subsection{Connection formula}

In this section we first express the discrete Sobolev-type polynomials 
$\left(s^{(\alpha)}_n(N;x)\right)$ in terms of the standard orthogonal 
polynomials $(p_n)$ and the kernel polynomials and its corresponding derivatives.
Taking into account the Fourier expansion, i.e., 
\[
s_n^{(\alpha)}(N;x) =p_n(x) +\sum_{k=0}^{n-1}a_{n,k}p_{k}(x).
\]
Then, from the properties of orthogonality of $\left(p_n\right)_n$ and 
$\left(s_n^{(\alpha)}(N;x)\right)_n$ respectively, we have
\begin{equation*}
a_{n,k}=\frac{\langle {\bf u}, s_n^{(\alpha)}(N;x)p_{k}(x)\rangle}{
\langle {\bf u}, p^2_{k}\rangle}=-\frac{N({\mathscr D}_q 
s_n^{(\alpha)})(N;\alpha)({\mathscr D}_q 
p_{k})(\alpha) }{\langle {\bf u}, p^2_{k}\rangle},\ \  0\leq k\leq n-1.
\end{equation*}
Thus we deduce  
\begin{equation} \label{3:5}
s_n^{(\alpha)}(N;x) =p_n(x) -N({\mathscr D}_q s_n^{(\alpha)})(N;\alpha)
{\mathscr K}_{q,n}^{(0,1) }(x,\alpha).
\end{equation}
\begin{remark}
Observe that $s_1^{(\alpha)}(N;x) =p_1(x)$. 
\end{remark}
From here, and after some basic manipulations, one gets
\[
({\mathscr D}_qs_n^{(\alpha)})(N;\alpha) =({\mathscr D}_qp_n)
(\alpha ) -N({\mathscr D}_qs_n^{(\alpha)})(N;\alpha)
{\mathscr K}_{q,n}^{(1,1) }(\alpha ,\alpha ),
\]
Therefore
\[
({\mathscr D}_q s_n^{(\alpha)})(N;\alpha) =\frac{({\mathscr D}_q p_n)(\alpha)}
{1+N{\mathscr K}_{q,n}^{(1,1)}(\alpha ,\alpha)},
\]
so from \eqref{3:5} one has
\begin{equation} \label{3:6}
s_n^{(\alpha)}(N;x) =p_n(x)-\frac{N({\mathscr D}_q p_n)(\alpha)}
{1+N{\mathscr K}_{q,n}^{(1,1) }(\alpha ,\alpha)}{\mathscr K}_{q,n}^{(0,1) }(x,\alpha).  
\end{equation}
\begin{remark}
From now on we denote $\frac{N({\mathscr D}_q p_n)(\alpha)}
{1+N{\mathscr K}_{q,n}^{(1,1)}(\alpha ,\alpha)}$
by $\mathscr C_n$.
\end{remark}
Finally taking into account the identity 
\[
\left({\mathscr D}_q\, \frac f{x-\alpha}\right)(x)=
\frac {({\mathscr D}_q f)(x)}{qx-\alpha}-\frac{f(x)}{(x-\alpha)(qx-\alpha)},
\]
one gets 
\begin{eqnarray}
{\mathscr K}_{q,n}^{(0,1) }(x,\alpha) &=&
\displaystyle \frac{p_n(x) p_{n-1}(\alpha)
-p_{n-1}(x) p_n(\alpha )}{\|p_{n-1}\|^2\,(x-\alpha)(x-q\alpha)} \nonumber\\
&&\displaystyle+\frac{p_n(x) ({\mathscr D}_qp_{n-1})(\alpha)
-p_{n-1}(x) ({\mathscr D}_qp_n)(\alpha)}{\|p_{n-1}\|^2\,(x-q\alpha)}.
\label{3:7}
\end{eqnarray}
With this, and by using expression \eqref{3:6}, one obtains
\begin{eqnarray}
s_n^{(\alpha)}(N;x) =p_n(x)-\frac{\mathscr C_n}{\| p_{n-1}\|^2}
\Bigg(\frac{p_n(x) p_{n-1}(\alpha )-p_{n-1}(x) p_n(\alpha ) }{(x-\alpha)
(x-q\alpha)} \nonumber \\ +\frac{p_n(x) ({\mathscr D}_qp_{n-1})(\alpha)
-p_{n-1}(x) ({\mathscr D}_qp_n)(\alpha)}{(
x-q\alpha )}\Bigg). \label{3:8}
\end{eqnarray}
Since the expression \eqref{3:6} is a rational function in $N$, 
where both the numerator and denominator has the same degree as $N$, 
it is clear that we can define the monic polynomial of 
degree $n$ which results on taking the limit $N\to \infty$. 
In fact, we obtain
\begin{equation}  \label{3:9}
r_{n}^{(\alpha)}(x):=\lim_{N\to\infty}s_n^{(\alpha)}(N;x)=p_n(x)
-\frac{({\mathscr D}_q p_n)(\alpha)}{{\mathscr K}_{q,n}^{(1,1)}(\alpha ,
\alpha)}{\mathscr K}_{q,n}^{(0,1)}(x,\alpha). 
\end{equation}
To characterize these new polynomials, first we observe that they are
strictly quasi-orthogonal of order $2$ with respect to the linear functional
\begin{equation*}
{\bf v}=(x-\alpha)(x-q\alpha){\bf u},
\end{equation*}
therefore, $r_{n}^{(\alpha)}(x)$ is a linear
combination of three consecutive polynomials of the sequence 
$(p_{n})$, i.e., for $n\geq 2$, we have
\[
r_{n}^{(\alpha)}(x)=p_n(x)+b_{q,n}p_{n-1}(x)+c_{q,n}p_{n-2}(x),
\]
where 
\[
c_{q,n}=\frac{\langle {\bf u}, r_{n}^{(\alpha)}(x)
p_{n-2}(x)\rangle}{\langle {\bf u}, p^2_{n-2}\rangle}=
-\frac{\left(({\mathscr D}_q p_n)(\alpha)\right)^2}{{\mathscr 
K}_{q,n}(\alpha,\alpha)}<0.
\]
\subsection{Distribution of the zeros}
Let $\left(\eta _{n,k}\right)_{k=1}^n$ be the zeros of $s_n^{(\alpha)}(N;x)$ 
and $\left(x_{n,k}\right)_{k=1}^n$ be the zeros of  $p_n(x)$. 
Then the following result holds.
\begin{proposition} \label{Prop:3.1}
(\cite{MPP-RdM92}, Proposition 6.2) The polynomial $r_n^{(\alpha)}(x)$ has 
$n$ real and simple zeros, namely $\left(y_{n,k}\right)_{k=1}^n$. \\
Moreover, if $\alpha<\text{supp}(\bf u)$, then 
\[
y_{n,1}<\alpha <x_{n,1}<y_{n,2}<x_{n,2}<\cdots <y_{n,n}<x_{n,n},
\]
and if $\alpha>\text{supp}(\bf u)$, then 
\[
x_{n,1}<y_{n,1}<x_{n,2}<\cdots <y_{n,n-1}<x_{n,n}<\alpha<y_{n,n},
\]
hold for every $n\geq 2$.
\end{proposition}
In order to obtain some results concerning monotonicity, asymptotics, and 
speed of convergence for the zeros in terms of the mass $N$ applying 
Lemma \ref{Lemm:2.1} we believe that is more convenient to normalize 
the connection formula \eqref{3:6} in this useful way:
\begin{proposition}
For this polynomial sequence the following identity holds:
\begin{equation}
\left(1+N{\mathscr K}_{q,n}^{(1,1)}(\alpha ,\alpha )\right)s_n^{(\alpha)}(N;x)=p_n(x)
+N{\mathscr K}_{q,n}^{(1,1)}(\alpha ,\alpha) r_{n}^{(\alpha)}(x).
\label{3:10}
\end{equation}
\end{proposition}
We leave the proof to the reader, it is just enough to replace \eqref{3:9}
in \eqref{3:10} and after some basic manipulations 
one gets the desired identity.

We point out the fact the basic hypergeometric-Sobolev-type orthogonal
polynomials $s_n^{(\alpha)}(N;x) $ appears as a linear
combination of two polynomials of degree $n$. 
Thus, from \eqref{3:10},  Proposition \ref{Prop:3.1}, and Lemma 
\ref{Lemm:2.1}, we immediately conclude that the mass point $\alpha$ 
does not attract any zero of  $s^{(\alpha)}_n(N;x)$ 
when $N\to \infty$, as in the standard case. 
By standard we mean the case of the polynomials orthogonal with
respect to the inner product \eqref{3:4} (see \cite{HuerMarcRaf10}).

Moreover, it is well-known that the polynomial $p_n(x)$ has $n$ different 
real zeros  and since we assumed that the interval $(\alpha, q \alpha)$ 
contains no points of the spectrum of $\bf u$, it follows that at most one 
zero of $p_n(x)$ is  situated in $[\alpha, q\alpha]$. 
For the polynomials $s_n^{(\alpha)}(N;x)$ we have the following result.
\begin{proposition}
If $n\ge 3$ the polynomial $s_n^{(\alpha)}(N;x)$ has at least $n-2$ 
different real zeros with odd multiplicity.
\end{proposition}
\begin{proof}
Let $\eta_{n,1}, \eta_{n,2}, \cdots, \eta_{n,k}$ denote the real zeros of 
$s_n^{(\alpha)}(N;x)$ of odd multiplicity. 
Put $\Pi(x)=(x-\eta_{n,1})\cdots (x-\eta_{n,k})$ We have 
\[
\langle s_n^{(\alpha)}(N;x),(x-\alpha)(x-q \alpha) \Pi(x)\rangle_S=
\langle {\bf u}, s_n^{(\alpha)}(N;x)(x-\alpha)(x-q \alpha) \Pi(x)\rangle>0.
\]
Hence $\deg \Pi\ge n-2$.
\end{proof}

The position of the real zeros of $s_n^{(\alpha)}(N;x)$ can be localized 
by the following theorem.
\begin{theorem} \label{theo:3.2}
Suppose $({\mathscr D}_q p_n)(\alpha)=0$, $N>0$. Let $k$ denote 
the intersection with the real axis of the chord which joins the points 
$(\alpha, p_n(\alpha))$ and $(q \alpha,p_n(q\alpha))$. 
\begin{enumerate}
\item If $k\not \in [x_{n,i},x_{n,i+1}]$ and $x_{n,i}$, $x_{n,i+1} \not \in 
[\alpha, q\alpha]$, then $(x_{n,i},x_{n,i+1})$ contains at least one zero 
of $s_n^{(\alpha)}(N;x)$.
\item If there exists a unique $0\le i\le n$ such that 
$\alpha<x_{n,i}<q\alpha$, then 
also $\alpha<k<q\alpha$ and  we have one of the following cases:
\begin{enumerate}
\item If $\alpha<k<x_{n,i}<q\alpha$ then $(x_{n,i-1},x_{n,i})$ contains at 
least one zero of $s_n^{(\alpha)}(N;x)$.
\item If $\alpha<x_{n,i}<k<q\alpha$ then $(x_{n,i},x_{n,i+1})$ 
contains at least one zero of $s_n^{(\alpha)}(N;x)$.
\end{enumerate}
\end{enumerate}
\end{theorem}
The proof of this result is analogue to the one presented in 
\cite[Theorem 3.2]{bav1} so we leave it to the reader.
\begin{remark}
Notice that depending on the sign of $\alpha$ the interval changes but 
the result is clear in both cases $\alpha>0$ and $\alpha<0$. 
\end{remark}
\begin{corollary}
The position of at least $n-2$ zeros of $s_n^{(\alpha)}(N;x)$ can be 
localized.
\end{corollary}
\begin{proof} We consider different cases:
\begin{itemize}
\item If $[\alpha,q\alpha]$ does not contain a zero of $p_n(x)$, since 
$x_{n,i}\ne \alpha$ for all $i$, then for at least $n-2$ intervals 
$(x_{n,i},x_{n,i+1})$ part 1. of Theorem \ref{theo:3.2} can be applied.
\item If $(\alpha,q\alpha)$ contains a zero of $p_n(x)$ (more than 
one is impossible), then $n-3$ zeros of $s_n^{(\alpha)}(N;x)$ can be 
localized by the first part  of Theorem \ref{theo:3.2} and one by the 
second part. 
\end{itemize}
Hence the result follows. \end{proof}

\subsection{The five-term recurrence relation}
Since the inner product \eqref{3:4} does not satisfy the 
Hankel property , i.e., 
\[
\langle x\, f, g\rangle_S = \langle f, x\, g\rangle_S,
\] 
the polynomial sequence $\left(s_n^{(\alpha)}(N;x)\right)_n$ 
does not fulfil a three-term recurrence relation. 
However, we find that 
\[
\langle (x-\alpha)(x-\alpha q) p, r\rangle_S 
= \langle p, (x-\alpha)(x-\alpha q) r\rangle_S,\quad 
p, r\in \mathbb P.
\]
Let us state the first result which is a direct consequence of 
this fact.
\begin{proposition}  
 The following identity holds for $n\ge 2$:
\begin{equation} \label{3:11}
(x-\alpha)(x-\alpha q)s_n^{(\alpha)}(N;x)=\sum_{\nu=n-2}^{n+2} 
a_{n,\nu} p_\nu(x),
\end{equation}
where $a_{n,n+2}=1$, and
\begin{eqnarray*}
a_{n,n+1}&=& \big(\beta_{n+1}+\beta_n-\alpha(1+q)\big)
-\frac{{\mathscr C}_n}
{\langle {\bf u}, p^2_{n-1}\rangle}
\left({\mathscr D}_q p_{n-1}\right)(\alpha), \\
a_{n,n}&=& \big(\gamma_{n+1}+\gamma_n + (\beta_{n+1}-\alpha q)
(\beta_{n}-\alpha)\big)\big)\\ && -
\frac{{\mathscr C}_n}{\langle {\bf u}, p^2_{n-1}\rangle}\big(p_{n-1}(\alpha)
-\left({\mathscr D}_q p_{n}\right)(\alpha)+(\beta_{n}-\alpha) 
\left({\mathscr D}_q p_{n-1}\right)(\alpha)\big),\\
a_{n,n-1}&=&  \gamma_n \big(\beta_{n}+\beta_{n-1}-\alpha(1+q)\big)
+\frac{{\mathscr C}_n}{\langle {\bf u}, p^2_{n-1}\rangle}
\big(p_n(\alpha)-\gamma_n \left({\mathscr D}_q 
p_{n-1}\right)(\alpha)\\ & & +(\beta_{n-1}-\alpha) \left({\mathscr 
D}_q p_n\right)(\alpha)\big),\\
a_{n,n-2}&=& \gamma_n  \gamma_{n-1} +
\frac{{\mathscr C}_n}{\langle {\bf u}, p^2_{n-1}\rangle}
\left({\mathscr D}_q p_n\right)(\alpha).
\end{eqnarray*}
\end{proposition}
Indeed, one can obtain all the coefficients of \eqref{3:11} by using 
the three-term recurrence relation of $(p_n)_n$ 
\[
xp_n(x)=p_{n+1}(x)+\beta_n p_n(x)+\gamma_n p_{n-1}(x),
\] 
with initial conditions $p_{-1}(x)=0$, $p_{0}(x)=1$,
the expression \eqref{3:6}, the identity 
\[
a_{n,\nu}\langle u, p_{\nu}^2\rangle=
\langle s_n^{(\alpha)}(N;x),(x-\alpha)(x-\alpha q)p_\nu\rangle_S,
\quad \nu=0, 1, \dots, n+2,
\]
and using the expansion of $(x-\alpha)(x-\alpha q) p_n(x)$, i.e.
\[\begin{split}
(x-\alpha)(x-\alpha q)& p_n(x)=  p_{n+2}(x)+\big(\beta_{n+1}
+\beta_n-\alpha(1+q)\big)p_{n+1}(x)+\big(\gamma_{n+1}+\gamma_n \\
+ & (\beta_{n+1}-\alpha q)(\beta_{n}-\alpha)\big)
p_{n}(x)+\gamma_n\big(\beta_{n}+\beta_{n-1}-\alpha(1+q)
\big)p_{n-1}(x)\\
+ & \gamma_n\gamma_{n-1}p_{n-2}.
\end{split}
\]

We now derive a recurrence relation for the polynomials 
$s_n^{(\alpha)}(N;x)$.
\begin{proposition}  (Five-term 
recurrence relation)   \\ 
 The following recurrence Relation holds for $n\ge 2$:
\begin{equation} \label{3:12}
(x-\alpha)(x-\alpha q)s_n^{(\alpha)}(N;x)=\sum_{\nu=n-2}^{n+2} 
\lambda_{n,\nu} s_\nu^{(\alpha)}(N;x),
\end{equation}
where $\lambda_{n,n+2}=1$, and
\[
\lambda_{n,j}=\dfrac{a_{j,n}\langle {\bf u}, p^2_{n}\rangle\left(1
+N{\mathscr K}^{(1,1)}_{q,n}
(\alpha,\alpha)\right)+N ({\mathscr D}_q p_n)(\alpha)\displaystyle
\sum_{i=n+1}^{j+2}a_{j,i}({\mathscr D}_q p_i)(\alpha)}{(1+N 
{\mathscr K}^{(1,1)}_{q,n}(\alpha,\alpha))\|S_j\|_S^2},
\]
where 
\[
\|S_n\|_S^2:=(S_n,S_n)_S=\langle s_n^{(\alpha)}(N;x), s_n^{(\alpha)}(N;x)
\rangle_S=\frac{1+N {\mathscr K}^{(1,1)}_{q,n+1}(\alpha,\alpha)}
{1+N {\mathscr K}^{(1,1)}_{q,n}(\alpha,\alpha)}\, \langle {\bf u}, p^2_{n}\rangle.
\]
\end{proposition}
The proof of this result is analogous to the one of  
the Proposition 3.1 in \cite{bav2}.
\begin{remark}
Notice that after basic manipulations of \eqref{3:12}we get
\[
\lambda_{n,n-2}=\frac{\|S_n\|_S^2}{\|S_{n-2}\|_S^2},
\qquad 
\lambda_{n,n+1}\|S_{n+1}\|_S^2=\lambda_{n+1,n}\|S_{n}\|_S^2.
\]
\end{remark}
 Observe that this result is direct after some straightforward calculations.
\subsection{The second order linear $q$-difference holonomic equation}
In the following we assume that $\bf u$ is a classical $q$-discrete functional.
Let $\Phi(x)$ denote the polynomial $(x-\alpha)(x-\alpha q)$. 
From the expression \eqref{3:11} we get 
\begin{equation} \label{3:13}
\Phi(x) s_n^{(\alpha)}(N;x)=A(x;n)p_n(x)+B(x;n)p_{n-1}(x),
\end{equation}
where $A(x; n)$ and $B(x; n)$ are polynomials of degree bounded by a 
number independent of $n$ and at most 2 and 1, respectively. 
On the other hand, since $\bf u$ is a classical linear functional, then 
there exist a polynomial $\Psi(x)$ and two polynomials $M(x;n)$ and 
$N(x;n)$, with degree bounded by a number independent of $n$, such 
that
\begin{equation} \label{3:14}
\Psi(x) ({\mathscr D}_q p_n)(x)=M(x;n)p_n(x)+N(x;n)p_{n-1}(x).
\end{equation}
Using  \eqref{3:13} and \eqref{3:14} we obtain the following representation 
formula:
\begin{equation} \label{3:15}
\pi(x) s_n^{(\alpha)}(N;x)=a(x;n)p_n(x)+b(x;n)p_{n}(q x),
\end{equation}
where $a$, $b$ and $\pi$ are polynomials of degree bounded by a 
number independent of $n$. With all these expressions we can 
formulate the result.
\begin{theorem} 
Let $\bf u$ be a classical linear functional. Suppose that the polynomials 
$\left(s_n^{(\alpha)}(N;x)\right)$ are defined by \eqref{3:15} where the 
polynomial  $p_n$ is a solution of a second order linear $q$-difference holonomic equation, $q$-SODE in short, of the form 
\begin{equation} \label{3:16}
\sigma(x;n)p_n(q^{-1}x)-\varphi(x;n)p_n(x)+\zeta(x;n) p_n(qx)=0.
\end{equation}
Then $\left(s_n^{(\alpha)}(N;x)\right)$ satisfy a $q$-SODE of the form
\begin{equation} \label{3:17}
\tilde \sigma(x;n)s_n^{(\alpha)}(N;q^{-1}x)-\tilde \varphi(x;n)
s_n^{(\alpha)}(N;x)+\tilde \zeta(x;n) s_n^{(\alpha)}(N;qx)=0,
\end{equation}
where $\tilde \sigma$, $\tilde \varphi$ and $\tilde \zeta$ 
can be computed explicitly.
\end{theorem}
A completely analogous proof is given in \cite[\S 3.2]{alpe} 
so it will be omitted.
\begin{remark}
It is clear that after some manipulations one can obtain some lowering and 
raising operators, namely ${\mathfrak a}^\dag$ and $\mathfrak a$. 
In fact such operators can be written as follows:
\[
\mathfrak a^\dag={\mathscr A}(x;n) {\mathscr D}_q+{\mathscr E}(x;n) I_d,
\quad 
\mathfrak a={\mathscr B}(x;n) {\mathscr D}_{q^{-1}}+{\mathscr F}(x;n) I_d,
\]
where $I_d$ represents the identity operator.
\end{remark}
\section{The examples}
\subsection{The Al-Salam-Carlitz I polynomials}
The Al-Salam-Carlitz I polynomials, ASCI in short, $U_n^{(a)}(x;q)$ were 
introduced in \cite{ASC1965} by Al-Salam and Carlitz (1965) and they are defined via 
basic hypergeometric series as \cite{koekoeketal} 
\begin{equation}
U_n^{(a)}(x;q):=\left(-a\right) ^{n}q^{\binom{n%
}{2}}\,_{2}\phi _{1}\left(q^{-n},x^{-1}; 0; q,a^{-1}qx\right).  \label{4:18}
\end{equation}
\begin{proposition}
For this polynomial sequence the following identity holds:
\begin{enumerate}
\item Orthogonality relation.  For $a<0$ 
\[ 
\int_{a}^{1}U_{m}^{(a)}(x;q) U_n^{\left( a\right)
}(x;q) \left( qx,a^{-1}qx;q\right) _{\infty }d_{q}x=\big\|
U_n^{(a)}\big\| ^{2}\delta _{m,n}.  
\] 
where by $\delta _{i,j}$ we denote the Kronecker delta function.
\item Squared norm. 
\[ 
\big\|U_n^{(a)}\big\| ^{2}=\left(1-q\right)\left(-a\right)
^{n} q^{\binom{n}{2}} \left( q;q\right) _{n}\left( q,a,a^{-1}q;q\right)
_{\infty }.
\] 

\item The Three-Term Recurrence Relation. For $n\ge 0$, 
\begin{equation} \label{4:19}
U_{n+1}^{(a)}(x;q)\!=\!\left(x\!-\!(a\!+\!1)q^{n}\right)\! U_n^{\left( a\right)}(x;q)
+aq^{n-1}\!\left(1\!-\!q^{n}\right)\!U_{n-1}^{(a)}(x;q),
\end{equation}
with initial conditions $U_{0}^{(a)}( x;q)=1$, and $U_{-1}^{(a)}(x;q)=0$.

\item Forward Shift Operator.
\begin{equation}  \label{4:20}
\left( {\mathscr D}_{q}U_n^{(a)}\right) (x;q)
=\left[ n\right]_{q}U_{n-1}^{(a)}(x;q).
\end{equation}%
\end{enumerate}
\end{proposition}
\subsubsection{Connection formula and hypergeometric representation}
We can express the ASCI\---Sobolev-type
polynomials $\left(U_n^{(a,\alpha)}(N;x;q)\right) $
in terms of the ASCI the associated Kernel polynomials. 
In fact, by construction, we have
\begin{equation} \label{DASMSM}
U_n^{(a,\alpha)}(N;x;q) =U_n^{(a)}(x;q) -\frac{N[n]_qU_{n-1}^{(a)}(\alpha ;q)}
{1+N{\mathscr K}_{q,n}^{(1,1)}\left( \alpha
,\alpha \right)}{\mathscr K}_{q,n}^{(0,1)}(x,\alpha).  
\end{equation} 
And by using \eqref{3:6} and \eqref{4:20} we obtain
\[
U_n^{(a,\alpha)}(N;x;q) =A_n(x) U_n^{(a)}(x;q)
+B_n(x)U_{n-1}^{(a)}(x;q) , 
\]
where%
\[
A_n(x) =1-\frac{N[n]_q U_{n-1}^{(a)}(\alpha;q)
\left(U_{n-1}^{(a)}\left(\alpha ;q\right)+
(x-\alpha) [n-1]_qU_{n-2}^{(a)}
\left(\alpha ;q\right)\right)}{\big\| U_{n-1}^{(a)}\big\|
^{2}(x-\alpha)(x-\alpha q)
\left(1+N{\mathscr K}_{q,n}^{(1,1)}(\alpha,\alpha)\right)} ,  
\]
and%
\[
B_n(x)=\frac{N[n]_q U_{n-1}^{(a)}( \alpha ;q)
\left(U_n^{\left(a\right) }\left( \alpha ;q\right) 
+\left( x-\alpha \right) [n]_q U_{n-1}^{(a)}( \alpha ;q)\right)}{
\big\| U_{n-1}^{(a)}\big\|^{2}(x-\alpha)(x-\alpha q)\left(
1+N{\mathscr K}_{q,n}^{(1,1)}( \alpha ,\alpha)\right)}.
\]
Thus, taking into account \eqref{4:18} as well as the
identity%
\begin{equation} \label{4:21}
\left(q^{1-n};q\right) _{k}=\frac{q}{\left[ n\right] _{q}}
\left(\left[ n-1\right] _{q}-\left[ k-1\right] _{q}\right) 
\left( q^{-n};q\right) _{k},
\end{equation}%
we deduce%
\begin{multline*}
U_n^{(a,\alpha)}(N;x;q) =-\left(
-a\right) ^{n-1}q^{\binom{n}{2}-n+2}\frac{B_n(x)}{\left[ n\right] _{q}} \\
\times \sum_{k=0}^n\left( \left[ k-1\right] _{q}+\Theta _n(x) \right) 
\frac{\left( q^{-n};q\right)_{k}\left( x^{-1};q\right) _{k}}{\left(q;q\right) _{k}}\left(
a^{-1}qx\right) ^{k},
\end{multline*}%
where%
\[
\Theta _n(x) =\frac{aq^{n-2}%
\left[ n\right] _{q}A_n(x) }{%
B_n(x) }-\left[n-1\right] _{q}.
\]
In addition, we have%
\begin{equation} \label{4:21a}
\left[k-1\right] _{q}+\Theta _n(x) =\frac{1-\varphi _n(x) q^{-1}}
{\varphi _n(x) \left( 1-q\right)}
\frac{\left( \varphi _n(x);q\right) _{k}}{\left(\varphi _n(x) q^{-1};q\right) _{k}},
\end{equation}
where
\[
\varphi _n(x) =\frac 1{(1-q)\Theta_n(x)  +1}.
\]
Consequently
\begin{multline*}
U_n^{(a,\alpha)}(N;x;q) =-(-a) ^{n-1}q^{\binom{n}{2}-n+2}\frac{B_n(x) }{[n] _{q}}
\frac{1-\varphi_n(x) q^{-1}}{\varphi _n(x) (1-q) } \\
\times \sum_{k=0}^n\frac{( q^{-n};q) _{k}\left(
x^{-1};q\right) _{k}\left( \varphi_n(x);q\right) _{k}}
{\left( \varphi_n(x) q^{-1};q\right) _{k}}\frac{( a^{-1}qx) ^{k}}{%
\left( q;q\right) _{k}}.
\end{multline*}%
Thus we have proven the following identity for the 
ASCI-Sobolev-type polynomial.
\begin{theorem}
The ASCI-Sobolev-type polynomial has the following 
hypergeometric representation: %
\begin{multline}
U_n^{(a,\alpha)}(N;x;q) =(-a) ^{n}\frac{B_n(x)(1-\varphi _n(x)q^{-1}) 
q^{\binom{n}{2}-n+2}}{a[n] _{q}\varphi _n(x) (1-q) } \\
\times \,_{3}\phi _{2}\left( q^{-n},x^{-1},\varphi _n(x); 
0,\varphi _n(x) q^{-1};q,a^{-1}qx\right) .  \label{4:22}
\end{multline}
\end{theorem}
\subsubsection{Distribution of the zeros}
Our main aim in this Section is to study the location as well as 
to obtain results concerning the monotonicity and speed of 
convergence  of the zeros of the ASCI-Sobolev-type polynomial 
$U_n^{(a,\alpha)}(N;x;q)$. 
For this purpose we use Lemma \ref{Lemm:2.1}. 

\begin{lemma}
The polynomial ${\mathscr K}_{q,n}^{\left( 0,1\right) }\left(x,\alpha \right) $ 
has $n-1$ real and  simple zeros which interlace with the zeros of 
$U_n^{(a)}(x;q) $.
\end{lemma}

The proof is straightforward by applying \cite[Lemma 1.1]{Jordaan} to 
the expression \eqref{3:7} and we leave it to the reader.
\begin{theorem}
Taking into account the identity \eqref{3:10}. 
For any  $\alpha \in \mathbb{R}$, with $\alpha\not \in [a,1]$ and 
for any $N>0$ the zeros of $U_n^{(a,\alpha)}(N;x;q)$, $U^{(a)}_n(x)$ 
and  $r^{(a,\alpha)}_n(x;q)$ fulfil the following interlacing relations:
\begin{itemize}
\item If $\alpha<a$ then 
\begin{equation*}
y_{n,1}<\eta _{n,1}<x_{n,1}<y_{n,2}<\eta _{n,2}<x_{n,2}<
\cdots <y_{n,n}<\eta_{n,n}<x_{n,n}.
\end{equation*}
Moreover, each zero $\eta _{n,k}$ is a decreasing function in $N$, i.e., 
$\eta _{n,k}=\eta _{n,k}(N)$; and, for each $k=2,\ldots ,n$
\begin{equation*}
\lim_{N\to \infty}\eta _{n,1}(N)=\alpha, \quad 
\lim_{N\to \infty}\eta _{n,k}(N)=y_{n,k},
\end{equation*}%
and
\begin{equation*}
\lim_{N\to \infty}N\Big(\eta _{n,k}(N)-y_{n,k}\Big)=\dfrac{-U^{(a)}_n(y_{n,k};q)}
{\left(r_{n}^{(a,\alpha)}(x;q)\right)'_{\big|x=y_{n,k}}}.
\end{equation*}%
\item If $\alpha>1$ then 
\begin{equation*}
x_{n,1}<\eta _{n,1}<y_{n,1}<x_{n,2}<\eta _{n,2}<y_{n,2}<
\cdots <x_{n,n}<\eta_{n,n}<y_{n,n}.
\end{equation*}%
Moreover, each zero $\eta _{n,k}$ is an increasing function in $N$; and, 
for each $k=1,\ldots ,n-1$
\begin{equation*}
\lim_{N\to \infty}\eta _{n,k}(N)=y_{n,k}, \quad 
\lim_{N\to \infty}\eta _{n,n}(N)=\alpha,
\end{equation*}%
and
\begin{equation*}
\lim_{N\to \infty}N\Big(\eta _{n,k}(N)-y_{n,k}\Big)=\dfrac{-U^{(a)}_n(y_{n,k};q)}
{\left(r_{n}^{(a,\alpha)}(x;q)\right)'_{\big|x=y_{n,k}}}.
\end{equation*}
\end{itemize}
\end{theorem}
Notice that the mass point $\alpha $ attracts one zero of 
\eqref{4:22}, i.e. when $N\rightarrow \infty $, it captures either 
the smallest or the largest zero, according to the location of the 
point $\alpha$ with respect to $[a,1] $. 
When either $\alpha <a$ or $\alpha>1 $, at most one of the 
zeros of \eqref{4:22} is located outside of $[a,1] $. 
Next, we give explicitly the value $N_{0}$ of the mass $N$, such 
that for $N>N_{0}$ one of the zeros is located outside $[a,1] $.

\begin{corollary}
\label{MMasa} If $\alpha \notin [a,1] ,$ the following statements
hold:

\begin{enumerate}
\item[i.)] if $\alpha <a$, then the smallest zero $\eta _{n,1}=\eta
_{n,1}(\alpha ) $ satisfies
\begin{equation*}
\begin{array}{c}
\eta _{n,1}>a,\ \mathrm{for}\ N<N_{0},\smallskip \\ 
\eta _{n,1}=a,\ \mathrm{for}\ N=N_{0},\smallskip \\ 
\eta _{n,1}<a,\ \mathrm{for}\ N>N_{0},
\end{array}
\end{equation*}
where $N_0=N_{0}(a,n,q,\alpha )$ with
\[
N_{0}=\left(\frac{[n]_qU_{n-1}^{(a)}(\alpha ;q)}
{U_n^{(a)}( a;q) }{\mathscr K}_{q,n}^{(0,1)}( a,\alpha) -
{\mathscr K}_{q,n}^{(1,1)}( \alpha ,\alpha)\right) ^{-1}>0.  
\]

\item[ii.)] if $\alpha >1$, then the largest zero $\eta _{n,n}=\eta
_{n,n}(\alpha ) $ satisfies
\begin{equation*}
\begin{array}{c}
\eta _{n,n}<1,\ \mathrm{for}\ N<N_{0},\smallskip \\ 
\eta _{n,n}=1,\ \mathrm{for}\ N=N_{0},\smallskip \\ 
\eta _{n,n}>1,\ \mathrm{for}\ N>N_{0},
\end{array}
\end{equation*}
with
\[
N_{0}=\left(\frac{[n]_qU_{n-1}^{(a)}(\alpha ;q)}{U_n^{(a)}(1;q) }
{\mathscr K}_{q,n}^{(0,1)}(1,\alpha) -{\mathscr K}_{q,n}^{(1,1)}
( \alpha ,\alpha) \right) ^{-1}>0.  
\]
\end{enumerate}
\end{corollary}

\begin{proof}
It suffices to use \eqref{3:6} together with the fact
that ${U_n^{(a,\alpha)}}\left(N;\tau ;q\right) =0$,
with $\tau =a,1$, if and only if $N=N_{0}$
\begin{equation*}
U_n^{(a,\alpha)}(N;\tau ;q)
=U_n^{(a)}\left( \tau ;q\right) -\frac{N_{0}\, U_n^{\left(
a\right) }\left( \alpha ;q\right) }{1+N_{0}{\mathscr K}_{q,n}^{(1,1)}
\left( \alpha ,\alpha \right) }{\mathscr K}_{q,n}^{(0,1)}\left( \tau
,\alpha \right) =0.
\end{equation*}
Thus
\begin{equation*}
U_n^{(a)}\left( \tau ;q\right) =\frac{N_{0}\,U_n^{\left(
a\right) }\left( \alpha ;q\right) }{1+N_{0}{\mathscr K}_{q,n}^{(1,1)}
\left( \alpha ,\alpha \right) }{\mathscr K}_{q,n}^{(0,1)}\left( \tau
,\alpha \right) .
\end{equation*}
Therefore
\begin{equation*}
N_{0}=N_{0}(a,n,q,\alpha )=\left(\frac{[n]_qU_{n-1}^{(a)}(\alpha ;q)}
{U_n^{(a)}(\tau;q) }{\mathscr K}_{q,n}^{(0,1)}( \tau,\alpha) -
{\mathscr K}_{q,n}^{(1,1)}( \alpha ,\alpha)\right) ^{-1}.
\end{equation*}
\end{proof}

Notice that, according to the well-known theorem of Hurwitz 
(see \cite{Frank46, Hur85}), for $n$ large enough,  only one zero 
of \eqref{4:22} is located outside of $\left[a,1\right]$  and it is 
attracted by $\alpha $. 
Next we show  some numerical experiments using Wolfram Mathematica 
software, dealing with  the smallest and the largest zero of \eqref{DASMSM}. 
We are interested to show the location and behaviour of these zeros.

In the first two tables we show the position for the first and last zero of 
\eqref{4:22} of degree $n=20$, $a=-1$ and $q=1/2$, for some 
choices of the mass $N$. Indeed, this case is connected to the discrete 
$q-$Hermite polynomials \cite{koekoeketal}. 
For $N=0$ obviously we recover the first zero and the last zero of 
\eqref{4:18}. 
When the mass point is located at $\alpha =-75<a$, then  
$N_0=3.12758\times 10^{-128}$; and when it is located 
at $\alpha =72>1$, then  $N_0=1.41561\times10^{-127}$ we obtain
\begin{equation*}
\begin{tabular}{ccccc}
\hline
\multicolumn{1}{l}{$N$} & $7.0\times 10^{-120}$ & 
$7.0\times 10^{-118}$ & $7.0\times 10^{-116}$
& $7.0\times 10^{-114}$  \\ \hline
\multicolumn{1}{l}{$\eta _{20,1}$} & $-63.6640$ &
$-74.8668$ &  $-74.9987$ & $-75.0001$ \\
\multicolumn{1}{l}{$\eta _{20,20}$} & $40.5829$ & $71.447$ & 
$71.9945$ & $72.$ \\ \hline
\end{tabular}
\end{equation*}
The next table shows the smallest and largest zeroes in the case 
when we set $n=11$ and $a=-5$, and the mass point is located at 
$\alpha =-42$  ($N_0=8.84157\times 10^{-41}$) and 
at $\alpha =45$ ($N_0=6.30851\times10^{-41}$) respectively
\begin{equation*}
\begin{tabular}{ccccc}
\hline
\multicolumn{1}{l}{$N$} & $7.0\times 10^{-36}$ & $7.0\times 10^{-34}$
& $7.0\times 10^{-32}$ & $7.0\times 10^{-30}$ \\ \hline
\multicolumn{1}{l}{$\eta _{11,1}$} & $-25.9916$ & $-41.7589$ & 
$-42.0138$ & $-42.0164$ \\
\multicolumn{1}{l}{$\eta _{11,11}$} & $41.4986$ & $44.9875$ & 
$45.0253$ & $45.0257$  \\ \hline
\end{tabular}
\end{equation*}
In the next tables we provide numerical evidences in support of 
Corollary \ref{MMasa}, where the exact values of $N_{0}$ are 
calculated for the $a=-1$ case. 
For this purpose we begin by analyzing the smallest zero of 
\eqref{4:22} of degree $n=14$, $q=1/2$, and with the
mass point located at $\alpha =-3$. In this case we have that
$N_{0}=N_{0}(-3,14,1/2)=5.13942\times 10^{-41}$. 
Thus
\begin{equation*}
\begin{tabular}{ccccc}
\hline
$N$ & $5.0 \times 10^{-45}$ & $N_{0}\left( -3,14,1/2\right) $ & 
$5.0\times 10^{-35}$ & $5.0\times 10^{-30}$ \\ \hline
$\eta _{14,1}$ & $-0.9999$ & $-1.0000$ & $-2.8906$ & $-3.0001$ \\ 
\hline
\end{tabular}
\end{equation*}
Finally, for the case where the mass point is located at $\alpha =21$, 
 we have $N_{0}\left(21,14,1/2\right) =4.03796\times 10^{-62}$ and%
\begin{equation*}
\begin{tabular}{ccccc}
\hline
$N$ & $6.0\times 10^{-65}$ & $N_{0}\left( 21,14,1/2\right) $ & 
$6.0\times 10^{-55}$ & $6.0\times 10^{-50}$ \\ \hline
$\eta _{14,14}$ & $0.9999$ & $1.0000$ & $20.7490$ & $21.0013$ \\ 
\hline
\end{tabular}
\end{equation*}
\subsection{The Stieltjes--Wigert polynomials}
The (monic) Stieltjes--Wigert polynomials, SW in short, $S_n(x;q)$ are 
defined via basic hypergeometric series as \cite{koekoeketal} 
\[
S_n(x;q):=(-1) ^{n}q^{-n^2}
{}_{1}\phi _{1}\left(q^{-n};0; q,-q^{n+1}x\right).  
\]
The moment problem for the SW polynomials is indeterminate; in 
other words, there are many different measures giving the same 
family of orthogonal polynomials (Krein's condition, see e.g. \cite{krein}).
\begin{proposition}
For this polynomial sequence the following identities hold:
\begin{enumerate}
\item Orthogonality relation%
\[
\int_{0}^{\infty}S_{m}(x;q) S_{n}(x;q) \,
\frac{dx}{(-x,-qx^{-1};q) _{\infty }}=\| S_n\| ^{2}\delta _{m,n}.  
\]
\item Squared Norm. 
\[
\| S_{n} \| ^{2}=-q^{-n(2n+1)}\, (q;q)_\infty (q;q)_n \log q.
\]
\item The recurrence relation. For $n\ge 0$, 
\[
\hspace{-3mm}S_{n+1}(x;q)=(x-q^{-2n-1}(1+q-q^{n+1}))S_{n}(x;q)-q^{-4n+1}
(1-q^n) S_{n-1}(x;q),
\]
with initial conditions $S_{0}(x;q)=1$, and $S_{-1}(x;q)=0$.
\item Forward shift operator 
\[
\left({\mathscr D}_q S_{n}\right)(x;q) = q^{-2(n-1)} [n] _{q} 
S_{n-1}(xq^2;q).  
\]
\end{enumerate}
\end{proposition}
\subsection{Connection formula and hypergeometric representation}
We denote by $\left(S^{(\alpha)}_{n}(N;x;q) \right)$ the sequence of monic 
polynomials orthogonal with respect to the inner product
\begin{equation*}
\langle f,g\rangle_S:=\int_{0}^{\infty} f(x) g(x)\, \frac{dx}{(-x,-qx^{-1};q) _{\infty }}
+N\left({\mathscr D}_qf\right)(\alpha)
\left({\mathscr D}_qg\right)(\alpha ).
\end{equation*}
These polynomials are connected with the Stieltjes--Wigert 
polynomials by the formula
\begin{equation}
S^{(\alpha)}_{n}(N;x;q) =S_n(x;q)-\frac{N({\mathscr D}_q S_{n})
(\alpha )}{1+N{\mathscr K}_{q,n}^{(1,1) }(\alpha ,\alpha )}
{\mathscr K}_{q,n}^{(0,1) }(x,\alpha) .  \label{4:29}
\end{equation}
Next, using the same idea than with the another example
and by using the identities \eqref{4:21} and \eqref{4:21a}, we 
get the following result.
\begin{theorem} The SW-Sobolev-type polynomial has 
the following hypergeometric  representation %
\[
S_n^{(\alpha)}(N;x;q) \!=\!
\frac{B_n(x)(1-\varphi_n(x)q^{-1})}{(q;q)_n (q^{-n}-1)} 
\,_{2}\phi _{2}\left( q^{-n},\varphi _n(x); 
0,\varphi _n(x) q^{-1} ;q,-q^{n}x\right),  
\]
where, in this case, 
\[
\varphi_n(x) =\frac{B_n(x)-A_n(x)(q^{n}-1)}
{B_n(x)}\, q^{1-n}.
\]
\end{theorem}
\subsubsection{Distribution of the zeros}
As in the previous example our main aim, again, 
in this Section is to study the location as well as 
to obtain results concerning the monotonicity and speed of 
convergence  of the zeros of the Stieltjes--Wigert-Sobolev-type 
polynomial, $S_n^{(\alpha)}(N;x;q)$. 

Notice that, when $\alpha \in \mathbb{R}$, with $\alpha<0$, at 
most  one of the zeros of $S_n^{(\alpha)}(N;x;q)$ is located 
outside $[0,\infty) $. 
Next we provide the explicit value $N_{0}$ of the mass such 
that for $N>N_{0}$ this situation appears, i.e, one of the zeros 
is located outside $[0,\infty) $.
\begin{theorem}
Taking into account the identity \eqref{3:10}. 
For any  $\alpha \in \mathbb{R}$, with $\alpha<0$ and 
for any $N>0$ the zeros of $S_n^{(\alpha)}(N;x;q)$, $S_n(x;q)$ 
and $r^{(\alpha)}_n(x;q)$ fulfil the following interlacing relations:
\begin{equation*}
y_{n,1}<\eta _{n,1}<x_{n,1}<y_{n,2}<\eta _{n,2}<x_{n,2}<
\cdots <y_{n,n}<\eta_{n,n}<x_{n,n}.
\end{equation*}
Moreover, each zero $\eta _{n,k}$ is a decreasing function in $N$, i.e., 
$\eta _{n,k}=\eta _{n,k}(N)$; and, for each $k=2,\ldots ,n$
\begin{equation*}
\lim_{N\to \infty}\eta _{n,1}(N)=\alpha, \quad 
\lim_{N\to \infty}\eta _{n,k}(N)=y_{n,k},
\end{equation*}%
and
\begin{equation*}
\lim_{N\to \infty}N\Big(\eta _{n,k}(N)-y_{n,k}\Big)=\dfrac{-S_n(y_{n,k};q)}
{\left(r_{n}^{(\alpha)}(x;q)\right)'_{\big|x=y_{n,k}}}.
\end{equation*}
\end{theorem}

\begin{corollary}
If $\alpha \in \mathbb{R}$, with $\alpha<0$, then 
the smallest zero $\eta _{n,1}=\eta _{n,1}(\alpha )$ satisfies 
\begin{equation*}
\begin{array}{c}
\eta _{n,1}>0,\ \mathrm{for}\ N<N_{0},\smallskip \\ 
\eta _{n,1}=0,\ \mathrm{for}\ N=N_{0},\smallskip \\ 
\eta _{n,1}<0,\ \mathrm{for}\ N>N_{0},%
\end{array}%
\end{equation*}%
where%
\[
N_{0}=N_{0}(\alpha ,n,q)=\left(\frac{{\mathscr D}_qS_{n}(\alpha;q)}
{S_{n}(0;q)}{\mathscr K}_{q,n}^{(0,1) }(0,\alpha)
-{\mathscr K}_{q,n}^{(1,1)}(\alpha,\alpha)\right)^{-1}>0.
\]
\end{corollary}

The proof of this result is analogue to the previous example so we 
leave it  to the reader. 
The numerical behaviour of the zeros for this family is very similar to 
the Al-Salam-Carlitz case, hence we are going to show 
here how changes the smallest zero of degree $n=14$
when $\alpha=-1$ for some real values of $0<q\le 1$ going to 1.

\begin{equation*}
\begin{tabular}{ccccccc}
\hline
$q$ & $N_{0}\left( 21,14,1/2\right)$ & $N$ & $10^{-5} N_0$ & $N_{0} $ & 
$10^{5} N_0$ & $10^{10} N_0$ \\ \hline
0.5 & 0.2717 & $\eta _{14,1}$ & $1.2482$ & $0.0000$ & $-3.2380$ 
& $-3.2381$ \\ 
0.8 & $2.95394\times 10^{-8}$ & $\eta _{14,1}$ & $0.5574$ 
& $0.0000$ & $-1.4878$ & $-1.4878$  \\ 
0.9 & $8.29261\times10^{-17}$ & $\eta _{14,1}$ & $0.4657$ & $0.0000$ 
& $-1.2625$ & $-1.2625$  \\ 
0.99 & $1.42958\times10^{-97}$ & $\eta _{14,1}$ & $0.6206$ & $0.0000$ 
& $-1.1597$ & $-1.1597$  \\ 
\hline
\end{tabular}
\end{equation*}

\subsection*{Acknowledgments}

The author R. S. Costas-Santos acknowledges financial 
support by Direcci\'on General de Investigaci\'on, 
Ministerio de Econom\'ia y Competitividad of Spain, 
grant MTM2015-65888-C4-2-P. The authors thank the 
anonymous referee for her/his valuable comments and 
suggestions. They contributed to improve the presentation 
of the manuscript.

\end{document}